\documentclass{svproc}
\usepackage{url}

\usepackage{amsmath}
\usepackage{graphicx}
\usepackage{amsfonts}
\usepackage{amssymb}
\usepackage{color}
\usepackage{algorithmic}
\usepackage{algorithm}
\usepackage{latexsym,theorem}
\usepackage{amsmath,amsfonts,amssymb,amsbsy}
\usepackage{subcaption}

\newcommand{\tr}{\top}
\newcommand{\Uset}{\mathbb{U}}
\newcommand{\Xset}{\mathbb{X}}
\newcommand{\Yset}{\mathbb{Y}}
\newcommand{\Rset}{\mathbb{R}}

\newcommand{\Nset}{\mathbb{N}}

\newcommand{\cU}{\mathcal{U}}
\newcommand{\cX}{\mathcal{X}}

\newcommand{\cV}{\mathcal{V}}

\newcommand{\cH}{\mathcal{H}}
\newcommand{\cZ}{\mathcal{Z}}
\newcommand{\cT}{\mathcal{T}}
\newcommand{\cY}{\mathcal{Y}}

\newcommand{\bu}{\mathbf{u}}

\newcommand{\by}{\mathbf{y}}
\newcommand{\bg}{\mathbf{g}}

\newcommand{\bk}{\mathbf{k}}

\newcommand{\rankk}{\operatorname{rank}}

\newcommand{\col}{\operatorname{col}}

\newcommand{\Span}{\operatorname{span}}

\begin{document}
\mainmatter              
\title{A universal reproducing kernel Hilbert space for learning nonlinear systems operators}
\titlerunning{A universal RKHS for learning nonlinear systems operators}  
%
\author{Mircea Lazar\inst{1}}
\authorrunning{Mircea Lazar} 
%
\tocauthor{Mircea Lazar}
\institute{Eindhoven University of Technology, Eindhoven 5612AP, The Netherlands \email{m.lazar@tue.nl}}

\maketitle             

\begin{abstract}
In this work, we consider the problem of learning nonlinear operators that correspond to discrete-time nonlinear dynamical systems with inputs. Given an initial state and a finite input trajectory, such operators yield a finite output trajectory compatible with the system dynamics. Inspired by the universal approximation theorem of operators tailored to radial basis functions neural networks, we construct a class of kernel functions as the product of kernel functions in the space of input trajectories and initial states, respectively. We prove that for positive definite kernel functions, the resulting \emph{product reproducing kernel Hilbert space} is dense and even complete in the space of nonlinear systems operators, under suitable assumptions. This provides a universal kernel-functions-based framework for learning nonlinear systems operators, which is intuitive and easy to apply to general nonlinear systems.  
\end{abstract}
\section{Introduction}
\label{sec1}
The field of data-driven modeling and control has recently received an increased interest due to advances in artificial intelligence and computing hardware, but also due to new applications of the fundamental lemma \cite{Willems_2005_fundamental} in data-driven simulation, prediction, and control \cite{Coulson2019,Berberich2020,Markovsky_2023,Markovsky_survey}. The current effort is focused on data-driven modeling and control of nonlinear systems and one of the popular approaches in this field is based on kernel functions and reproducing kernel Hilbert spaces, see the 
comprehensive surveys \cite{Pillonetto_Kerneles_2014} and \cite{Martin_2023_SOS}. In general, we encounter methods that use kernel functions to learn \emph{functions} or \emph{mappings} and methods that use kernel functions to learn \emph{operators}.

Within the first category mentioned above, kernel functions have been used to learn/model the system dynamics with the purpose of designing data-driven (predictive) controllers, in \cite{Lian_Jones_2021,Claudio_kernels,RoKDeePC}. In particular, \cite{RoKDeePC} uses kernels to parameterize multi-step input-output predictors for data-enabled predictive control (DeePC) \cite{Coulson2019} and minimizes the multi-step prediction error when computing the interpolation coefficients. The kernel functions are defined in a lifted space that includes both initial conditions (past inputs and outputs) and an input sequence of finite length in time (given by the prediction horizon). Other relevant works include \cite{Maddalena_kernels}, which provides optimized bounds for kernel-based approximations, and the recent work \cite{Molodchyk_2024exploring}, which explores links between the fundamental lemma \cite{Willems_2005_fundamental} and kernel regression for linear systems and specific classes of nonlinear systems, i.e., Hammerstein and flat nonlinear systems. In particular, \cite{Molodchyk_2024exploring} formulates the kernel regression problem for vector valued mappings and extends standard bounds on the approximation error \cite{Fasshauer2011PositiveDK,Claudio_kernels} to singular Gramians. 

Alternatively, a discrete-time dynamical system can be regarded as an operator defined on a Hilbert space \cite{Henk_kernel_2023} that maps sequences of inputs to sequences of outputs (or input functions of time to output functions of time, in the continuous-time case).  In \cite{Henk_kernel_2023}, the operator learning problem is formulated as a kernel regression problem that includes knowledge of physical properties of the operator, via integral quadratic constraints. Also, \cite{Henk_kernel_2023} provides a comprehensive summary of results on reproducing kernel Hilbert spaces for vector-valued operators. More recently, \cite{Shali_Henk_2024representer} continues in this framework by considering kernel-based learning of monotone operators related to passive dynamical systems.   

Two fundamental results on universal approximation of nonlinear operators originate much earlier in the  fundamental works of Tianping Chen and Hong Chen, i.e., \cite{Chen_ONN}, which considers deep neural networks, and \cite{Chen_RBF}, which considers radial basis functions neural networks.These works have been recently brought into attention by the development of the deep operator network (DeepONet) \cite{DeepONet} for learning nonlinear operators, with application to continuous-time dynamical systems and partial differential equations. \cite{DeepONet} also provides a bound on the number of samples of the input signals required to achieve a desired approximation error. More recently, \cite{Batlle_2024_Kernels_Competitive} developed a kernel-based framework for learning operators defined on vector-valued reproducing kernel Hilbert spaces, associated with partial differential equations. Therein, conditions for asymptotic convergence to zero of the approximation error were established. 

In this paper, we consider the problem of learning a nonlinear system operator using the framework of reproducing kernel Hilbert spaces (RKHS) \cite{Bishop_2006,Steinwart_2008,Fasshauer2011PositiveDK} and we consider the following questions: which assumptions regarding the set of data points, the class and construction of kernel functions, and the topology of the corresponding Hilbert space are necessary to achieve a RKHS that is dense and even complete in the space of operators of interest. Specifically, we focus on operators arising in discrete-time nonlinear systems described by difference equations with initial states, inputs, and outputs restricted to compact sets. We explicitly parameterize output trajectories as functions of input trajectories \emph{and} initial states (or initial conditions in a more general sense). As such, we define the domain of the operator as a \emph{product} Hilbert space and the corresponding kernel functions are defined as the product of kernel functions in the input trajectories space and the state space, respectively. Based on the universal approximation theorem for operators using radial basis functions neural networks \cite{Chen_RBF}, we prove that for positive definite kernels \cite{Micchelli_2006} the resulting \emph{product} RKHS is dense in the space of nonlinear systems operators and its unique minimizer is a universal approximator.  

The developed product RKHS framework for learning nonlinear systems operators brings important contributions with respect to existing frameworks, such as, \emph{(i)}~regarding \cite{DeepONet,Chen_RBF}, it provides a universal operator learning framework that does not require training a large neural network, i.e., the interpolation coefficients are simply obtained by inverting the product RKHS Gram matrix; \emph{(ii)}~regarding \cite{RoKDeePC}, the developed product RKHS framework scales much better with the number of data points in terms of building and inverting the Gram matrix.

\section{Preliminary definitions and results}
\label{sec2}
Let $\Rset$ and $\Nset$ denote the sets of real and natural numbers, respectively. For any finite number $q\in\Nset_{\geq 1}$ of vectors (or functions) $\{v_1,\ldots,v_q\}$, we will make use of the operator $\col(v_1,\ldots,v_q):=[v_1^\tr,\ldots,v_q^\tr]^\tr$. For two matrices $A\in\Rset^{n\times n}, B\in\Rset^{m\times m}$, $A\otimes B\in\Rset^{nm\times nm}$ denotes their Kronecker product. For two vectors $a\in\Rset^n, b\in\Rset^m$, $a\otimes b\in\Rset^{nm}$ denotes their Kronecker vector product, i.e., for $a=[a_1\, a_2]^\top$, $b=[b_1\, b_2]^\top$, $a\otimes b=[a_1b_1\, a_1b_2 \, a_2b_1 \, a_2b_2]^\top$.

As the data generating system, we consider an unknown discrete-time nonlinear system with input $u:\Nset\rightarrow \Rset^m$, output $y : \Nset\rightarrow \Rset^p$, and state $x:\Nset \rightarrow \Rset^n$. We assume that all measurable signals are noise free and for ease of notation we use the state $x$ to denote either a minimal, true state, or a non-minimal one, i.e., represented by past inputs and outputs. For some finite $N\in\Nset_{\geq 1}$, we denote the corresponding Hilbert spaces $\cU:=\ell_2(\{0,1,\ldots,N\},\Rset^m)$, $\cY:=\ell_2(\{0,1,\ldots,N\},\Rset^p)$ and $\cX:=\ell_2(\{0\},\Rset^n)$, where the notation was inspired by \cite{Shali_Henk_2024representer}. Note that these are spaces of square summable sequences by definition.

The considered operator learning problem is formulated as follows: Given a set of input trajectories $\{\bu_1,\ldots,\bu_{T_u}\}$ with $\bu_i\in\cU$, a set of initial states $\{x_1,\ldots,x_{T_x}\}$ with $x_j\in\cX$,  and a corresponding set of output trajectories \[\{\by_1^1,\ldots,\by_{T_u}^1,\ldots,\by_1^{T_x},\ldots\by_{T_u}^{T_x}\},\quad \by_i^j\in\cY,\] 
find an accurate approximation of the operator $G(\bu)(x)=\by\in\Rset^{p(N+1)}$, with respect to some suitable norm. The operator $G$ maps sequences $\bu\in\cU$ into sequences $\by\in\cY$ for every $x\in\cX$ and it is assumed that the sequences $\by$ are continuous functions of $x$. Above $T_u\in\Nset_{\geq 1}$ is the number of input trajectories and $T_x\in\Nset_{\geq 1}$ is the number of initial states. Let $[G(\bu)(x)]_q$ denote the $q$-th element of $G(\bu)(x)$, with $q=1,\ldots, p(N+1)$.
\subsection{Universal approximation of operators based on radial basis functions neural networks}
\label{sec2.1}
Next we recall the fundamental operator approximation result for radial basis functions (RBF) neural networks \cite[Theorem~5]{Chen_RBF} adapted to the specific setting of this paper, i.e., \emph{(i)}~discrete-time dynamical systems, which means the input and output signals are already sampled in time; \emph{(ii)}~sampling in the space of initial conditions; and \emph{(iii)}~vector valued operators. This result will be instrumental in the construction of a \emph{dense} RKHS in the next section.

First we need to recall the notation in \cite{Chen_RBF}. Let $g\in C(\Rset)\cap S'(\Rset)$ denote a suitable radial basis function, where $C(\Rset)$ denotes the Banach space of all continuous functions defined on $\Rset$ with norm $\|f\|_{C(\Rset)}=\max_{s\in\Rset}|f(s)|$. Furthermore, $S'(\Rset)$ denotes the set of all linear functional defined on $S(\Rset)$, i.e., the set of infinitely differentiable functions which are rapidly decreasing at infinity. This definition covers most typical RBFs, but even polynomials are excluded \cite{Chen_RBF}. 
\begin{theorem}\cite[Theorem~5]{Chen_RBF}\label{thm:Chen_and_Chem}
Suppose that $g\in C(\Rset)\cap S'(\Rset)$ is not an even polynomial, $\Uset\subseteq \cU$, $\Xset\subseteq\cX$ are two compact sets in $\cU$ and $\cX$, respectively and $G$ is a nonlinear continuous operator, which maps $\Uset$ into $C(\Xset)$. Then for any $\epsilon>0$ there are positive integers $T_u, T_x$, constants $c_{iq}^j, \mu_i, \lambda_j\in\Rset$, $i=1,\ldots,T_u$, $j=1,\ldots,T_x$, $q=1,\ldots,p(N+1)$, points $\bu_1,\ldots,\bu_{T_u}\in\Uset$, $x_1,\ldots, x_{T_x}\in\Xset$, such that 
\begin{equation}
\label{eq:chen}
\left |[G(\bu)(x)]_q-\sum_{i=1}^{T_u}\sum_{j=1}^{T_x}c_{iq}^j g\left(\mu_i\|\bu-\bu_i\|\right)g\left(\lambda_j\|x-x_j\|\right)\right|<\epsilon
\end{equation}
for all $\bu\in \Uset$, all $x\in\Xset$ and all $q=1,\ldots,p(N+1)$.
\end{theorem}
Note that \cite[Theorem~5]{Chen_RBF} considers input and output trajectories continuous in time and sampled in time, so therein $x$ represents continuous time and $x_j$ are time samples (although sampling in state is also mentioned as a possibility). Since in the discrete-time case $\bu$ and $\by$ are sampled trajectories (sequences) by definition, it is no longer necessary to sample in time; instead we consider sampling in the state space to cover all output trajectories generated by a set of initial conditions. In \cite[Theorem~5]{Chen_RBF} $\cU$ must be a Banach space, which is true in our case since $\cU$ is a Hilbert space. 

It is also important to point out that $g\in C(\Rset)\cap S'(\Rset)$ rather denotes a class of functions in \cite[Theorem 5]{Chen_RBF} instead of a fixed function. I.e., the universal approximation theorem holds with different functions for the input sequences space (i.e., $g_u$) and the state space (i.e., $g_x$), respectively, as long as these functions belong to  $C(\Rset)\cap S'(\Rset)$ and they are not an even polynomial. In fact, \cite[Remark 4]{Chen_RBF} states that radial basis functions could be mixed with affine basis functions in the universal approximation theorem.

Another relevant aspect is that \cite[Theorem~5]{Chen_RBF} considers real-valued operators. However, since the approximation theorem holds for every separate element of the operator, we can construct a common set of points in $\Uset$ and $\Xset$ by aggregating all the points that exist for all elements and then we obtain \eqref{eq:chen} by setting some of the $c_{iq}^j$ equal to zero, for each $q$-th element. Also, \cite[Theorem~5]{Chen_RBF} uses a different set of points $\bu_i$ for each point $x_j$, i.e., $\bu_{ij}$. By using a similar aggregate construction, we can group all the $\bu_{ij}$ points for all $j$ in a single set of points $\bu_i$ and then set some of the coefficients $c_{iq}^j$  equal to zero for each $j$ to obtain \eqref{eq:chen}.

Note that the radial basis functions $g$ used in \cite[Theorem 5]{Chen_RBF} can be alternatively defined as kernel functions. For example, let $s:=x-x_j$, $\lambda:=\frac{1}{\sigma^2}$ and consider a Gaussian radial basis function:
\begin{equation}
\label{eq:rbf:ker}
g(\lambda \|s\|)=e^{-\lambda\|s\|^2}=e^{-\frac{1}{\sigma^2}\|s\|^2}=e^{-\frac{1}{\sigma^2}\|x-x_j\|^2}=k(x,x_j),
\end{equation}
i.e., we have obtained the Gaussian kernel function. This analogy will be instrumental in the next section for constructing a product RKHS. To this end, next we introduce some basic RKHS definitions.

\subsection{Reproducing kernel Hilbert spaces}
\label{sec2.2}
Compared to the operator learning approaches using RKHS in \cite{Henk_kernel_2023,Shali_Henk_2024representer}, for the setting of this paper, i.e., discrete-time dynamical systems and Hilbert spaces of square summable sequences, we adopt a simplified approach. More specifically, we regard operators $G : \cX\rightarrow \cY$, where $\cX$ denotes a generic Hilbert space of functions in this subsection, and $\cY$ is the Hilbert space of output trajectories, as an aggregation of a finite number of functionals from $\cX$ to the set of real numbers. Moreover, if the elements of $\cX$ are square summable sequences of finite length that can take arbitrary values in a suitable vector space, the functionals can be regarded as functions. As such we can make use of the standard RKHS theory for real-valued functions/functionals \cite{Fasshauer2011PositiveDK} and we can adopt a similar approach as in \cite{Molodchyk_2024exploring}, i.e., we can assumes a common RKHS for all functions/functionals. To this end we recall the following standard RKHS definitions, see, e.g., \cite{kernels_technical_report,Fasshauer2011PositiveDK}. 
\begin{definition}
\label{kerdef}
A function $k : \cX \times \cX\rightarrow \Rset$ is called a \emph{kernel function} if it satisfies the following properties: \emph{(i)} it is symmetric, i.e., $k(x_1,x_2)=k(x_2,x_1)$ for all $(x_1, x_2)\in\cX\times\cX$; and \emph{(ii)} it is positive semi-definite, i.e., for every positive integer $T$ and distinct set of points $\{x_1,\ldots,x_T\}\in\cX$ the matrix
\begin{equation}
\label{eq:gram}
K:=\begin{pmatrix}k(x_1,x_1) & k(x_1,x_2) &\ldots&k(x_1,x_T)\\k(x_2,x_1) & k(x_2,x_2) &\ldots&k(x_2,x_T)\\ \vdots & \vdots & \ddots & \vdots\\ k(x_T,x_1) & k(x_T,x_2) &\ldots&k(x_T,x_T)\\\end{pmatrix}\in\Rset^{T\times T}
\end{equation}
is positive semi-definite. The matrix $K$ is called the \emph{Gram matrix}. Moreover, a kernel function $k$ is called a \emph{universal kernel} (or a \emph{positive definite kernel}) if its corresponding Gram matrix is positive definite.
\end{definition}

Two important classes of universal kernels \cite{Micchelli_2006}, which are also radial basis functions that comply with the conditions of \cite[Theorem 5]{Chen_RBF} are Gaussian kernels, 
\begin{equation}
\label{eq:gauss:k}
k(x_1,x_2):=e^{-\frac{1}{\sigma^2}\|x_1-x_2\|^2}
\end{equation}
and Hardy reverse multiquadratics kernels, 
\begin{equation}
\label{eq:hardyR:k}
k(x_1,x_2):=\left(1+\frac{1}{\sigma^2}\|x_1-x_2\|^2\right)^{-\frac{1}{2}},
\end{equation}
where $\sigma$ is a positive real number.
\begin{definition}
\label{defRKHS}
Given a kernel function $k : \cX\times\cX\rightarrow\Rset$, a Hilbert space $\cH(k,\cX)$ is a \emph{reproducing kernel Hilbert space} for $k$ if \emph{(i)} for every $x\in\cX$, the function $k(x,\cdot)\in\cH(k,\cX)$ and \emph{(ii)} the reproducing property holds, i.e., $f(x)=\langle f, k(x,\cdot) \rangle_{\cH(k,\cX)}$ for every $f\in\cH(k,\cX)$ and $x\in\cX$. 
\end{definition}

In data-based learning (or fitting) problems \cite{Fasshauer2011PositiveDK}, typically a data-dependent finite dimensional Hilbert space is considered, i.e., \[\cH(k,\cX)=\Span\{k(\cdot, x_1), k(\cdot, x_2),\ldots,k(\cdot, x_T)\},\] where $\{x_1,\ldots,x_T\}$ are distinct data points and $T\geq 1$ is sufficiently large. In what follows we will make use of the following notation for the vector-kernel corresponding to the standard basis, i.e., $\Span\{k(\cdot, x_1), k(\cdot, x_2),\ldots,k(\cdot, x_T)\}$:
\begin{equation}
\label{eq:vec:ker}
\bk(x):=\col(k(x_1,x), k(x_2,x),\ldots,k(x_T,x))\in\Rset^T, \quad\forall x\in\cX.
\end{equation}
Given a set of observations (measurements) $\{y_1,\ldots,y_T\}$ with each $y_i\in\Rset^{n_y}$ corresponding to the set of data points 
$\{x_1,\ldots,x_T\}$ with each $x_j\in\cX$, the operator learning problem can be formulated in the RKHS $\cH(k,\cX)$ as \cite{Bishop_2006,Molodchyk_2024exploring}
\[\min_{\Theta\in\Rset^{T\times n_y}}\|Y^\top-K\Theta\|_F^2,\]
where $\|M\|_F$ denotes the Frobenius norm of a matrix $M$, $K$ is the corresponding Gram matrix and $Y=\begin{pmatrix}y_1&y_2&\ldots&y_T\end{pmatrix}\in\Rset^{n_y\times T}$. If $k$ is a positive definite kernel and hence $K$ is invertible, there exists a unique solution corresponding to the unique interpolated surface passing through all the points $y_i$, which yields the optimal approximation  
\begin{equation}
\label{eq:minimizer}
G_{\cH(k,\cX)}^\ast(x):=(\Theta^\ast)^\top\bk(x)=Y K^{-1}\bk(x), \quad \forall x\in\cX.
\end{equation}

Through the prism of Willems' fundamental lemma \cite{Willems_2005_fundamental}, we can write down the following system of equations for any $x\in\cX$:
\begin{subequations}
\label{eq:fundk}
\begin{align}
K\bg&=\bk(x)\label{eq:fundka}\\
Y\bg&= G_{\cH(k,\cX)}(x),\label{eq:fundkb}
\end{align}
\end{subequations}
where $\bg=\bg(x)\in\Rset^T$ is a vector of optimization variables (not to be confused with the RBF function $g$ from the previous subsection). I.e., for every $x\in\cX$, there should exist a $\bg$ that satisfies  \eqref{eq:fundk}. We observe that the first block of equations in \eqref{eq:fundka} correspond to property $\emph(i)$ of the RKHS, i.e., for any $x\in\cX$, it must hold that $\bk(x)\in\cH(k,\cX):=\Span\{\bk(x_1), \bk(x_2),\ldots,\bk(x_T)\}$. Furthermore, for a positive definite kernel, when $K$ is invertible, the second block of equations in \eqref{eq:fundkb} ensures that 
\[G_{\cH(k,\cX)}^\ast(x)=Y\bg^\ast=YK^{-1}\bk(x)=\langle G,\bk(x)\rangle_{\cH(k,\cX)},\]
which is the reproducing property. 

\subsection{Problem formulation}
Several existing works have exploited physical properties of the operator \cite{Henk_kernel_2023,Shali_Henk_2024representer} or of the underlying system dynamics/functions \cite{Maddalena_kernels,RoKDeePC,Molodchyk_2024exploring} to define kernel functions and corresponding RKHS. Differently, in this paper, we do not make specific assumptions about the class of the underlying nonlinear dynamics and we ask the following fundamental questions:
\begin{itemize}
\item[Q1] Does there exist a class of kernel functions $k$ such that the resulting RKHS $\cH(k,\cX)$ is dense in the considered space of nonlinear systems operators ?
\item[Q2] Does there exist a class of kernel functions $k$ such that the resulting RKHS $\cH(k,\cX)$ is complete with respect to the considered space of nonlinear systems operators ?
\end{itemize}
In the next section, we provide a possible answer to the above questions and we show that construction of a universal RKHS for learning nonlinear systems operators is linked to how initial states or initial conditions are incorporated in the learning problem.

\section{Product reproducing kernel Hilbert spaces}
\label{sec3}
In what follows we adopt the following definitions of a dense and complete subspace in a Hilbert space from \cite[Chapter 3]{Yamamoto_2012}.
\begin{definition}
\label{def:dense}
Let $\cV=\Span\{v_i \ : \ i=1,2,\ldots\}$ be a subspace in a Hilbert space $\cX$. $V$ is \emph{dense} in $\cX$ if its closure agrees with the whole space $\cX$, i.e., if every point in $\cX$ is a closure point of $\cV$. 
\end{definition}
The above dense property can be regarded as a univeral approximation property similar to \eqref{eq:chen}, i.e., for every point $x\in\cX$ and any $\epsilon>0$, there exist points $\{v_1,\ldots, v_n\}$, $v_i\in\cV$ and coefficients $c_i$ such that $\|x-\sum_{i=1}^nc_iv_i\|\leq \epsilon$.  
\begin{definition}
\label{def:orto}
A familiy $\{e_n\}_{n=1}^\infty$ of non-zero elements in a Hilbert space $\cX$ is said to be an \emph{orthogonal system} if $\langle e_i,e_j\rangle=0$ for all $i\neq j$. If, further, $\|e_i\|=1$ for all $i$, then $\{e_n\}_{n=1}^\infty$ is called an \emph{orthonormal system}.
\end{definition}
\begin{definition}
\label{def:complete}
An orthonormal system $\{e_n\}_{n=1}^\infty$ in a Hilbert space $\cX$ is said to be \emph{complete} if it has the property that the only element orthogonal to all $e_n$ is $0$. 
\end{definition}

Recall next the considered operator learning problem, i.e., consider a set of input trajectories $\{\bu_i \ : \ i=1,2,\ldots\}$ with $\bu_i\in\cU$, a set of initial states $\{x_j \ : \ j=1,2,\ldots\}$ with $x_j\in\cX$ and a corresponding set of output trajectories \[\{\by_i^j \ : \ i=1,2,\ldots,\, j=1,2,\ldots\},\quad \by_i^j\in\cY,\] 
where $\cU,\cX,\cY$ are suitable Hilbert spaces. We would like to use these trajectories to learn the underlying operator $G(\bu)(x)$ that maps sequences in $\cU$ into sequences in $\cY$ for every point in $\cX$. To this end we will construct a \emph{product} RKHS as follows.

Let $k_u:\cU\times\cU\rightarrow\Rset$ and $k_x:\cX\times\cX\rightarrow\Rset$ be kernel functions, respectively. Further, for any set of points $\{\bu_i \ : \ i=1,2,\ldots\}$ with $\bu_i\in\cU$ and set of initial states $\{x_j \ : \ j=1,2,\ldots\}$ with $x_j\in\cX$
let $K_u, K_x$ denote the corresponding Gram matrices and let $\bk_u, \bk_x$ denote the corresponding vector-kernels. Then we can construct a product kernel function (e.g., see \cite[Lemma 4.6]{Steinwart_2008})
\begin{equation}
\label{eq:prok}
k_\otimes:(\cU\times\cX)\times(\cU\times\cX)\rightarrow \Rset,\quad k_\otimes((\bu_1,x_1),(\bu_2,x_2)):=k_u(\bu_1,\bu_2)k_x(x_1,x_2)
\end{equation}
with corresponding Gram matrix and vector-kernel, i.e.,
\begin{equation}
\label{eq:proGram}
K_\otimes=K_u\otimes K_x \quad\text{and}\quad \bk_\otimes(\bu,x)=\bk_u(\bu)\otimes\bk_x(x),\end{equation}
and unique \emph{product} RKHS $\cH(k_\otimes,\cU\times\cX)$.

We consider given a set of observations (measurements) $\{y_1^1,\ldots,y_{T_u}^{T_x}\}$ with $y_i^j\in\Rset^{n_y}$ (e.g., for an output trajectory from time $0$ to $N$, $n_y=p(N+1)$, where $p$ is the number of outputs) corresponding to the set of data points $\{u_1,\ldots,u_{T_u}\}$ with each $u_i\in\Rset^{n_u}$ (e.g., for an input trajectory from time $0$ to $N$, $n_u=m(N+1)$, where $m$ is the number of inputs) and $\{x_1,\ldots,x_{T_x}\}$ with each $x_j\in\Rset^{n_x}$.  Recall here that $x$ needs not be the true state, i.e., it can be formed from past inputs and outputs observations and the number of states $n_x$ depends on this choice. Then, letting $T:=T_uT_x$ the operator learning problem can be formulated in the \emph{product} RKHS $\cH(k_\otimes,\cU\times\cX)$ as 
\[\min_{\Theta\in\Rset^{T\times n_y}}\|Y^\top-K_\otimes\Theta\|_F^2,\]
where 
\[Y=\begin{pmatrix}y_1^1 & y_2^1 & \ldots & y_{T_u}^1  &  \ldots & y_1^{T_x} & y_2^{T_x} & \ldots & y_{T_u}^{T_x}\end{pmatrix}\in\Rset^{n_y\times T}.\]
The corresponding fundamental system of equations for any $(\bu,x)\in\cU\times\cX$ is:
\begin{subequations}
\label{eq:fundkpro}
\begin{align}
K_\otimes\bg&=\bk_\otimes(\bu,x)\label{eq:fundkproa}\\
Y\bg&= G_{\cH(k_\otimes,\cU\times\cX)}(\bu,x)\label{eq:fundkprob}
\end{align}
\end{subequations}
where $\bg\in\Rset^{T}$. If $K_\otimes\in\Rset^{T\times T}$ is full rank and thus invertible, the corresponding unique minimizer is obtained as 
\begin{equation}
\label{eq:proGstar}
\begin{split}
G_{\cH(k_\otimes,\cU\times\cX)}^\ast(\bu,x)&=(\Theta^\ast)^\top\bk_\otimes(\bu,x)=Y\bg^\ast\\&=YK_\otimes^{-1}\bk_\otimes(\bu,x)=\langle G,\bk_\otimes(\bu,x)\rangle_{\cH(k_\otimes,\cU\times\cX)}.
\end{split}
\end{equation}
We observe that the constructed product RKHS has a special topology, i.e., it corresponds to the product of two Hilbert spaces, i.e., the space of square summable input sequences of finite length and the space of initial states. This allows sampling these two spaces independently to generate the data set. 

Next, we state the main result, which establishes  conditions for universality of the constructed product RKHS. 
\begin{theorem}
\label{thm:uni:springer}
Suppose that $k_u,k_x\in C(\Rset)\cap S'(\Rset)$ are not even polynomials and that they are positive definite  kernel functions, e.g., Gaussian or reverse multiquadratic radial functions. Let $\Uset\times\Xset\subseteq\cU\times\cX$ be a compact set in $\cU\times\cX$ and let $G$ be a nonlinear continuous operator, which maps $\Uset$ into $\Yset\subset C(\Xset)$ for every $x\in\Xset$ and belongs to $\cH(k_\otimes, \Uset\times\Xset)$. Then the following results hold:
\begin{itemize}
\item[\emph{(i)}] For any $\epsilon>0$ there are positive integers $T_u, T_x$, and points $\{\bu_1,\ldots,\bu_{T_u}\}$, $\bu_i\in\Uset$, $\{x_1,\ldots,x_{T_x}\}$, $x_j\in\Xset$, such that 
\begin{equation}
\label{eq:thm1}
\|G(\bu)(x)-G_{\cH(k_\otimes,\Uset\times\Xset)}^\ast(\bu,x)\|_2< \epsilon,
\end{equation}
for all $\bu\in\Uset$ and all $x\in\Xset$, i.e., $\Span\{\{k_\otimes(\cdot, (u_i,x_j))\}_{i=1,\ldots,T_u,j=1,\ldots,T_x}\}$ is dense in the space of nonlinear systems operators that map $\Uset$ into $\Yset\subset C(\Xset)$.
\item[\emph{(ii)}] There exists a possibly infinite set of points $\{(\bu,x)_i \ : \ i=1,2,\ldots\}$, $(\bu,x)_i\in\Uset\times\Xset$, such that the $\Span\{\{k_\otimes(\cdot, (\bu,x)_i)\}_{i=1,2,\ldots}\}$ induces a complete orthonormal system in the space of nonlinear systems operators that map $\Uset$ into $\Yset\subset C(\Xset)$.
\end{itemize}
\end{theorem}
\begin{proof}
\emph{(i)} Recall the relation \eqref{eq:rbf:ker} between kernel functions $k(z_1,z_2)$ and radial basis functions $g(\lambda \|z_1-z_2\|)$. The parameter $\lambda>0$ is typically also present in kernel functions, but not explicitly mentioned and it is chosen independently of the data points. Since the universal approximation theorem \cite[Theorem 5]{Chen_RBF} and condition \eqref{eq:chen} are formulated in terms of data-dependent $\lambda_j, \mu_i$, we will first show that when the RBFs $g$ correspond to universal radial functions/kernel, inequality \eqref{eq:chen} holds with data-independent parameters $\lambda$ (and $\mu$). To this end, define a RBF function
$g\in(\cZ\times\cZ)\times\Rset_+\rightarrow\Rset$ for some suitable set $\cZ$, corresponding to a universal positive definite radial kernel function $k$ as $g(\lambda \|z_1-z_2\|):=k(z_1,z_2)$, which explicitly reveals the dependence on the parameter $\lambda$. For such functions, e.g., as defined in \eqref{eq:gauss:k}, \eqref{eq:hardyR:k}, it holds that $0<g(\lambda\|z_1-z_2\|)< g(\lambda \|z_1-z_1\|)=g(0)$ for any $(z_1,z_2)\in\cZ\times\cZ$ and $\lambda\in\Rset_+$, e.g. for Gaussian and reverse multiquadratics $g(0)=1$. Hence, for any $\lambda(z_2),\lambda\in\Rset_+$, there exists a $\bar\lambda(z_2)\in\Rset_+$ such that $g(\lambda(z_2)\|z_1-z_2\|)=\bar\lambda(z_2)g(\lambda\|z_1-z_2\|)$, i.e. $\bar\lambda(z_2):=\frac{g(\lambda(z_2)\|z_1-z_2\|)}{g(\lambda\|z_1-z_2\|)}$.

Next, define the RBF approximator corresponding to radial basis functions $g_u,g_x$ obtained from the positive definite radial kernel functions $k_u,k_x$, i.e., which are defined using parameters $\mu,\lambda\in\Rset_+$ (e.g., these parameters are typically chosen as the average distances between neighboring points over the whole data set; or they could be identified from data), respectively, i.e. 
\begin{equation}
\label{eq:proof:1}
[G_{RBF}(\bu,x)]_q:=\sum_{i=1}^{T_u}\sum_{j=1}^{T_x}\bar c_{iq}^j g_u(\mu\|\bu-\bu_i\|)g_x(\lambda\|x-x_j\|),
\end{equation}
where $\bar c_{iq}^j:=c_{iq}^j\bar\mu(\bu_i)\bar\lambda(x_j)$ for all $(i,j)$ and $q=1,\ldots,n_y$. Then by Theorem~\ref{thm:Chen_and_Chem} (recall that we can apply this theorem with different functions $g_u, g_x$ as long as they are in the right class of radial functions; choosing the same kernel function and the same $g$ is also possible), for any $\epsilon>0$ there exist positive integers $T_u, T_x$, corresponding set of points and coefficients $\bar c_{iq}^j$ such that 
\begin{equation}
\label{eq:proof:2}
\|G(\bu)(x)-G_{RBF}(\bu,x)\|_2^2<\epsilon,\quad \forall \bu\in\Uset,\, x\in\Xset.
\end{equation}
Indeed, we can apply Theorem~\ref{thm:Chen_and_Chem} to attain \eqref{eq:chen} with $\epsilon$ replaced by $\sqrt{\frac{\epsilon}{n_y}}$ to attain \eqref{eq:proof:2}. Then we can expand \eqref{eq:proof:1} and write out the RBF approximator as a function of the corresponding kernels $k_u, k_x$, i.e., 
\begin{equation}
\label{eq:proof:3}
\begin{split}
[G_{RBF}(\bu,x)]_q&=\sum_{i=1}^{T_u}\sum_{j=1}^{T_x}\bar c_{iq}^j g_u(\mu\|\bu-\bu_i\|)g_x(\lambda\|x-x_j\|)\\
&=\sum_{i=1}^{T_u}\sum_{j=1}^{T_x}\bar c_{iq}^j k_u(\bu,\bu_i)k_x(x,x_j)\\
&=\theta_q^\top \bk_\otimes(\bu,x),
\end{split}
\end{equation}
where 
$\theta_q = (\bar c_{1 q}^1 \, \ldots \, \bar  c_{T_u q}^{T_x})^\top\in\Rset^{T}$ for all $q=1,\ldots,n_y$. Hence we can build a matrix of coefficients $\Theta\in\Rset^{T\times n_y}$ with its $q$-th column equal to $\theta_q$ such that $G_{RBF}(\bu,x)=\Theta^\top\bk_\otimes(\bu,x)$.

Next, we observe that since $k_u$ and $k_x$ are positive definite kernels, their corresponding Gramians $K_u,K_x$ have full rank and are invertible for any set of distinct points in $\Uset$ and $\Xset$, respectively. Then, by the property of the Kronecker product, we have that $\rankk(K_\otimes)=\rankk(K_u)\rankk(K_x)$,
i.e., the product Gram matrix $K_\otimes$ also has full rank and it is invertible. This could also be established by proving that the product kernel $k=k_u k_x$ is positive definite when $k_u$ and $k_x$ are positive definite. Hence, the fundamental system of equations \eqref{eq:fundka} has a unique solution $\bg^\ast$ for each $(\bu,x)$, which yields 
\[G_{RBF}(\bu,x)=\Theta^\top\bk_\otimes(\bu,x)=\Theta^\top K_\otimes\bg^\ast(\bu,x),\quad \forall (\bu,x)\in\Uset\times\Xset.\] 
Hence, $G_{RBF}(\bu,x)$ is an interpolant in the product RKHS $\cH(k_\otimes, \Uset\times\Xset)$. Since  we assume that $G(\bu)(x)$ belongs to $\cH(k_\otimes,\Uset\times\Xset)$, it follows that \cite{Fasshauer2011PositiveDK} $G_{\cH(k_\otimes,\Uset\times\Xset)}^\ast(\bu,x)$ provides the best approximation of $G(\bu)(x)$ from $\cH(k_\otimes, \Uset\times\Xset)$ with respect to the 2-norm, which yields
\begin{equation}
\label{eq:uni:proof}
\|G(\bu)(x)-G_{\cH(k_\otimes,\Uset\times\Xset)}^\ast(\bu,x)\|_2\leq \|G(\bu)(x)-G_{RBF}(\bu,x)\|_2 < \epsilon,\end{equation} for all $(\bu,x)\in\Uset\times\Xset$, 
which completes the proof of statement \emph{(i)}.

To prove statement \emph{(ii)}, note that by the positive definiteness of the product kernel function $k_\otimes$, for any set of distinct points $\{(\bu,x)_i \ : \ i=1,2,\ldots\}$, $(\bu,x)_i\in\Uset\times\Xset$, the Gram matrix $K_\otimes$ is full rank and thus, the sequence of kernel-vectors $\{k_\otimes(\cdot, (u,x)_1),\ldots, k_\otimes(\cdot, (u,x)_n),\ldots\}$ is linearly independent, i.e., any finite subset is linearly independent. Then, one can choose an orthonormal family $\{e_1,\ldots,e_n,\ldots\}$ such that for every $n$, $\Span\{e_1,\ldots,e_n\}=\Span\{k_\otimes(\cdot, (u,x)_1),\ldots, k_\otimes(\cdot, (u,x)_n)\}$ via the Gram-Schmidt orthohonalization procedure, see, e.g., \cite[Chapter 3, Problem 1]{Yamamoto_2012}. Since in statement \emph{(i)} we have shown that $\Span\{k_\otimes(\cdot, (u,x)_1),\ldots, k_\otimes(\cdot, (u,x)_n)\}$ is dense in the space of nonlinear systems operators that map $\Uset$ into $\Yset\subset C(\Xset)$, it follows that $\Span\{e_n\ : \ n=1,2,\ldots\}$ is dense in this space as well. Then, by \cite[Proposition 3.2.19]{Yamamoto_2012} we have that the orthonormal system $\{e_n\}_{n=1}^\infty$ is complete. \qed
\end{proof} 

The above results show that the developed product RKHS offer a universal framework for learning operators arising in discrete-time nonlinear dynamical systems, which is intuitive and computationally efficient, i.e., it scales well with the number of data points in terms of building the product Gram matrix $K_\otimes$. See the illustrative example for numerical details. 
\begin{remark}
In the proof of Theorem~\ref{thm:uni:springer} we require the assumption that $G(\bu)(x)$ belongs to the product kernel reproducing Hilbert space $\cH(k_\otimes,\Uset\times\Xset)$, which corresponds to the kernel functions choice. This assumption is necessary for inequality \eqref{eq:uni:proof} to hold, because the RBF approximation of the operator, i.e., $G_{RBF}(\bu,x)$ defined in \eqref{eq:proof:1}, employs data-dependent interpolation coefficients. Removing this assumption will be considered in future work.  
\end{remark}
A similar product RKHS can be developed for continuous dynamical systems operators by additionally sampling in time, i.e., by considering a product Hilbert space $\cU\times\cX\times\cT$ with associated product kernel functions and RKHS. Alternatively, finitely parameterized input sequences could be pursued as done in \cite{Shali_Henk_2024representer} in combination with a product RKHS in $\cU\times\cX$. 

The developed results yield a new insight in persistency of excitation for nonlinear systems, i.e., that a persistently exciting input sequence should be paired with a linearly independent/distinct set of points in the space of initial conditions. This is not necessary for linear systems, but seems to play a key role for nonlinear systems as the product Gram matrix is full rank if and only if the Gram matrix in the input space and the Gram matrix in the space of initial conditions is full rank. Last but not least, the developed results can be extended to Gram matrices that are not full rank by using singular value decomposition or orthogonal projection to remove the data points that result in very small or zero singular values for the product Gram matrix. This is interesting because it would allow using any type of kernel functions, not just positive definite ones.

\section{Illustrative example}
\label{sec4}
Consider a discrete-time state space model of the Van der Pol oscillator 
\begin{align*}
x_1(k+1) &= x_1(k) + T_s x_2(k) \\
x_2(k+1) &= x_2(k) + T_s(\mu(1-x^2_1(k)) x_2(k)-x_1(k)+u(k))\\
y(k)&= x(k),
\end{align*}
with $\mu=1$ and sampling period $T_s=0.1$. The output is equal to the full state and we would like to learn an operator that predicts the output trajectories of the system $N=10$ steps ahead. We will implement the developed product RKHS framework and we will compare the results with the standard RKHS framework. The two approaches differ in terms of data generation and how they scale with the number of data points. As such, to ensure a reasonably fair comparison, we fine tune each approach to get the best possible results, while using the same type of kernel function and the same method for generating a persistently exciting input signal. The results are obtained using the same laptop computer (Lenovo ThinkPad X1, Intel vPro i7 processor) and Matlab 2023a.

To generate a persistently exciting input we use the Matlab function \emph{idinput}, a sum-of-sinusoidal type of input signal with Range as $[-5, 5]$, SineData as $[25, 40, 1]$, Band as $[0, 1]$, NumPeriod as $1$ and Nu as $1$. As the kernel function we use the Hardy reverse multiquadratic kernel \eqref{eq:hardyR:k} with $\sigma=\sqrt{4}$ for the standard kernel approach and, $\sigma_x=\sqrt{0.4}$ and $\sigma_u=\sqrt{2}$ for the product kernel approach.  

For the product kernel approach we select $T_x=150$ initial conditions (states) and we construct $T_u=290$ input sequences via the Hankel matrix approach from a generated persistently exciting input signal of sufficient length. The resulting ouput trajecetories used for operator learning are plotted in Figure~\ref{fig1}.
\begin{figure}[h]
\centering
    \includegraphics[width=0.8\linewidth]{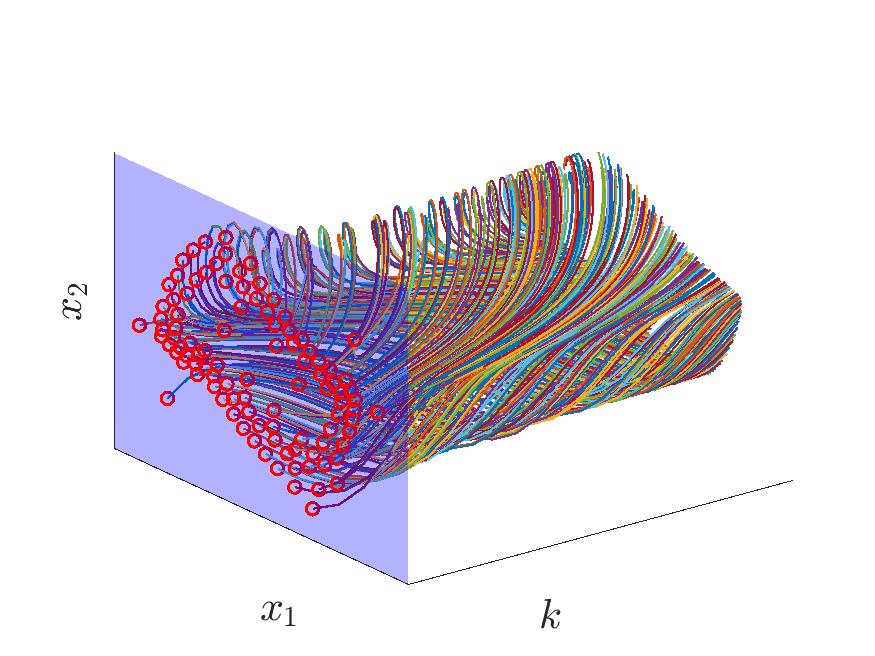}
 \caption{Output trajectories used for learning the operator in the product RKHS.} 
\label{fig1}
\end{figure}

Building the Gram matrices $K_u$ and $K_x$ requires $T_u^2+T_x^2=106600$ kernel function evaluations, which takes less than $5$ minutes. Then the inverse product Gram matrix $K^{-1}_\otimes$ is obtained by inverting $K_u$ and $K_x$ and computing their Kronecker product which takes less than $2$ seconds. By merging the initial states and the input sequences, the product RKHS approach generates $43500$ data points in the lifted $(x,u)$ space and a product Gram matrix $K_\otimes$ of dimension $43500^2$, i.e., roughly $1.8$ billions. Using so many data points in the standard RKHS approach would require $1.8$ billion kernel function evaluations just to build the Gram matrix, which is intractable/very time consuming. 

The prediction capability of the operator learned used the product RKHS approach is shown in Figure~\ref{fig2}.
\begin{figure}[h]
\centering
    \includegraphics[width=0.8\linewidth]{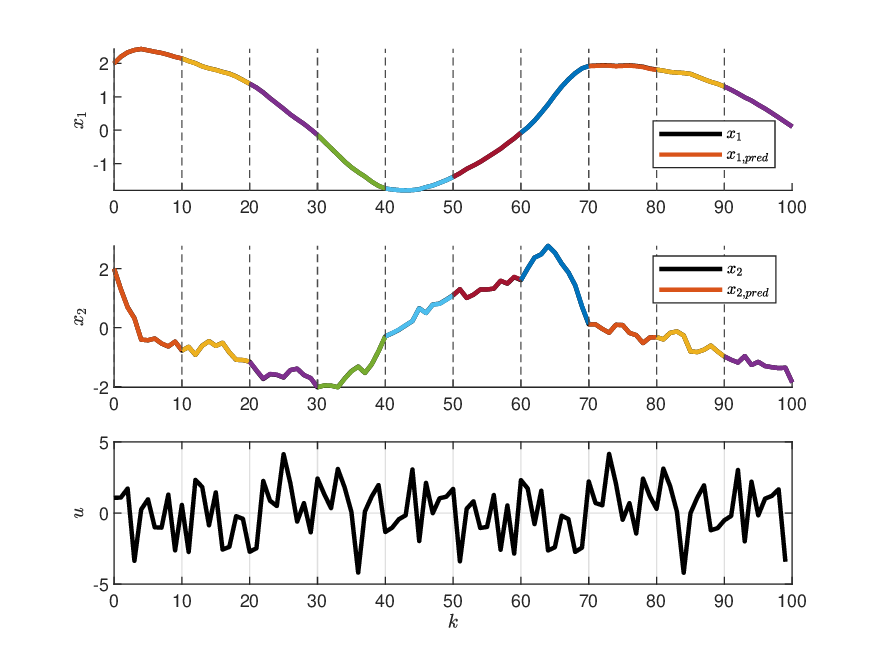}
 \caption{Validation of the product kernel operator predictions for several initial conditions not part of the training data and a random input signal.} 
\label{fig2}
\end{figure}

To test the standard RKHS approach, we generated a persistently exciting input signal of length $10000$, which results in $9990$ data points in the lifted space $(x,u)$ via the Hankel matrix approach. This yields a Gram matrix $K$ of dimension $9990^2$ or roughly $0.1$ billions. Building the Gram matrix took more than an hour and inverting it took $20$ seconds. Attempting to build the Gram matrix for a signal of length $20000$ did not return a result after several hours. The prediction capability of the mapping learned using the standard RKHS approach is shown in Figure~\ref{fig3}.
\begin{figure}[h]
\centering
    \includegraphics[width=0.8\linewidth]{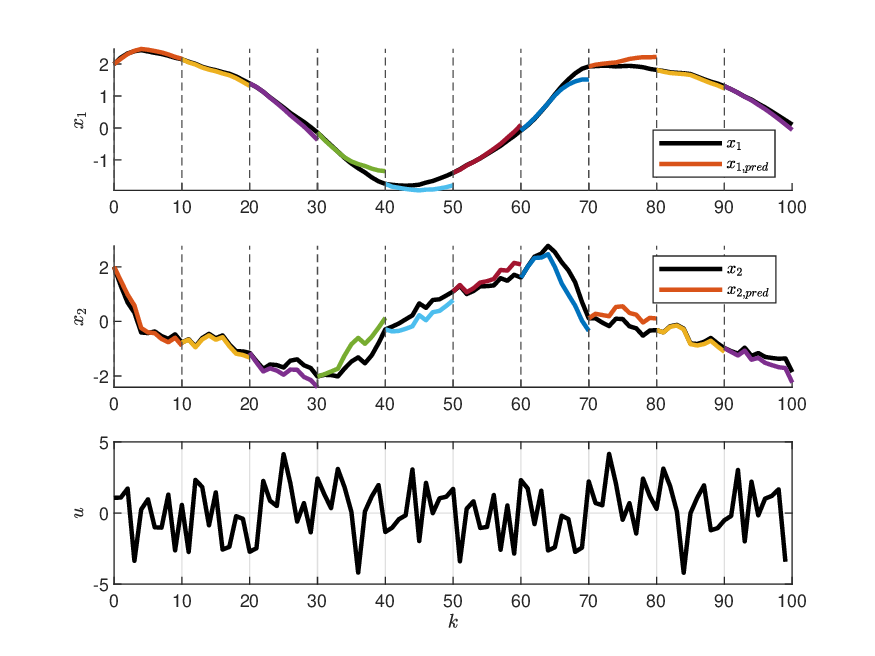}
 \caption{Validation of the standard kernel mapping predictions for the same initial conditions and random input signal used to validate the product kernel operator.} 
\label{fig3}
\end{figure}

To better compare the accuracy of the implemented learning methods, in Figure~\ref{fig4} we provide a Root Mean Square (RMS) prediction error analysis per each time step, for the first state $x_1$.
\begin{figure}[h]
\centering
    \includegraphics[width=0.8\linewidth]{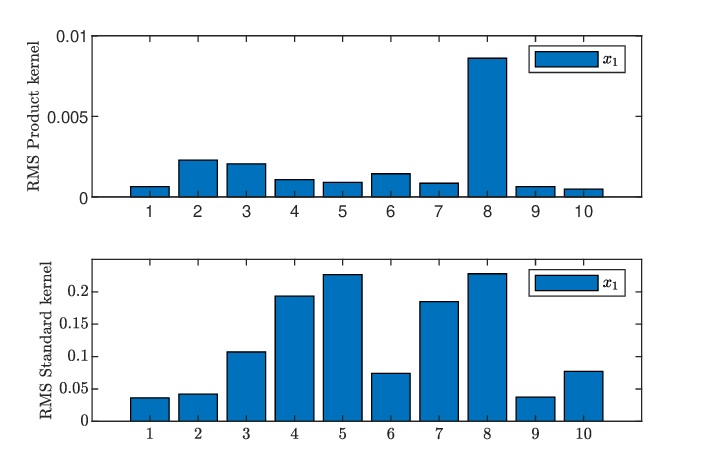}
 \caption{RMS for the prediction of $x_1$ for each time step $i=1,\ldots,10$.} 
\label{fig4}
\end{figure}
The obtained results demonstrate the efficiency and accuracy of the developed product RKHS framework for learning nonlinear systems operators.

\section{Conclusions}
\label{sec5}
Based on the universal approximation theorem of operators tailored to radial basis functions neural networks, we constructed a universal class of kernel functions as the product of kernel functions in the space of input sequences and initial states, respectively. We proved that for positive definite kernel functions, the resulting \emph{product reproducing kernel Hilbert space} is dense and even complete in the space of nonlinear systems operators that map input sequences and initial states into output trajectories. This provides a universal kernel-functions-based framework for learning nonlinear systems operators that is intuitive and easy to apply.

In future work we will present a relation between product RKHS operators and the Koopman operator with important implications for the nonlinear fundamental lemma. Also, we will exploit product RKHS operators to design data-enabled predictive controllers for nonlinear systems.
\paragraph{Acknowledgements}
The author would like to express his gratitude to the PhD researcher Thomas de Jong for building the simulation scripts and the figures for the illustrative example. The author is also grateful to the anonymous reviewers for their helpful comments.

\bibliographystyle{spmpsci}
\bibliography{Mircea}

\end{document}